\documentclass[a4paper,12pt,english]{amsart}
\usepackage{preamble}

\title{Generic isomorphism classes of abelian groups}

\author[U. B. Darji]{Udayan B. Darji}
\address{Department of Mathematics, University of Louisville, Louisville, KY 40292, USA}
\email{ubdarj01@louisville.edu}

\author[M. Elekes]{M{\'a}rton Elekes}
\address{Alfréd Rényi Institute of Mathematics, Reáltanoda u. 13-15, 1053 Budapest, Hungary, AND Institute of Mathematics, Eötvös Loránd University, Pázmány P. sétány 1/c, 1117 Budapest, Hungary
\url{http://www.renyi.hu/~emarci}}
\email{elekes.marton@renyi.hu}

\author[T. K{\'a}tay]{Tam{\'a}s K{\'a}tay}
\address{Institute of Mathematics, Eötvös Loránd University, Pázmány P. sétány 1/c, 1117 Budapest, Hungary}
\email{13heted@gmail.com}

\author[A. Kocsis]{Anett Kocsis}
\address{Institute of Mathematics, Eötvös Loránd University,
    Pázmány P. sétány 1/c, 1117 Budapest, Hungary}
\email{sakkboszi@gmail.com}

\author[M. P{\'a}lfy]{M{\'a}t{\'e} P{\'a}lfy}
\address{Alfréd Rényi Institute of Mathematics, Reáltanoda u. 13-15, 1053 Budapest, Hungary, AND Institute of Mathematics, Eötvös Loránd University, Pázmány P. sétány 1/c, 1117 Budapest, Hungary}
\email{palfymateandras@gmail.com}

\date{March 2023}

\begin{document}



\begin{abstract}
We prove that the universal solenoid is the generic (in the sense of Baire category) connected compact metrizable abelian group. We also settle the dual problem in the sense of Pontryagin duality: $(\rat,+)$, which is the dual of the universal solenoid, is the generic countably infinite torsion-free abelian group.
\end{abstract}

\maketitle

\section{Introduction}
\label{s.introduction}

\textbf{Genericity.} The generic behavior, in the sense of Baire category, has been extensively studied in various classes of mathematical objects for almost a century. For example, classical results describe generic properties of continuous functions \cite{BANACH_1931, BRUCKNER_GARG_1977, KIRCHHEIM_1995}, homeomorphisms \cite{OXTOBY_1937}, and continua \cite{BING_1951}. In recent decades, genericity has become the subject of increasing interest and intensive research. To mention only a handful of the large number of relevant works, homeomorphisms of compact spaces \cite{GLASNER_WEISS_2001, AKIN_HURLEY_KENNEDY_2003, KECHRIS_ROSENDAL_2007, AKIN_GLASNER_WEISS_2008, HOCHMAN_2008}, automorphisms of countable homogeneous structures \cite{TRUSS_1992, KECHRIS_ROSENDAL_2007}, relational structures \cite{CAMERON_1991, POUZET_ROUX_1996, KRUCKMAN_2016}, Polish metric and Banach spaces \cite{VERSHIK_2004, CUTH_DOLEZAL_DOUCHA_KURKA_2022}, Fraïsse limits \cite{KABLUCHKO_TENT_2017, KABLUCHKO_TENT_2017_2}, and group representations \cite{GLASNER_KITROSER_MELLERAY_2016, DOUCHA_MALICKI_2019} have been studied from the viewpoint of genericity. In some cases \cite{KUBIS_2016, KRAWCZYK_KUBIS_2021, BARTOS_KUBIS_2022}, genericity is defined and studied via an abstract Banach-Mazur game.

Surprisingly, groups, undoubtedly central objects in mathematics, had not been studied from this perspective until 2020. This line of research was initiated by I.~Goldbring, S.~Kunnawalkam Elayavalli, Y.~Lodha \cite{GOLDBRING}, and, independently, by M.~Elekes, B.~Gehér, K.~Ka\-na\-las, T.~Kátay, T.~Keleti, A.~Kocsis, and M.~Pálfy \cite{ELEKES, ELEKES22}. Among others, the authors of \cite{GOLDBRING} proved that every isomorphism class is meager in the space $\calg$ of countably infinite groups. They also studied the generic behavior in various subspaces of $\calg$ defined by natural group properties. The authors of \cite{ELEKES, ELEKES22} proved that the isomorphism class of the group $\bigoplus_{i\in\nat}(\rat/\int)$ is comeager in the subspace of abelian groups. They also initiated the study of generic properties of topological groups and proved the following.

\begin{theorem}\label{t.odometer}\cite{ELEKES22}
The isomorphism class of the countably infinite power of the universal odometer is comeager in the space of compact metrizable abelian groups.
\end{theorem}

The present paper is a successor of \cite{ELEKES22}. We study the generic behavior in two classes of abelian groups, which are dual to each other (in the sense of Pontryagin duality): connected compact metrizable abelian groups and torsion-free (discrete) countable abelian groups. It turns out that the former problem involves so-called indecomposable continua. The study of indecomposable continua is a classical branch of continuum theory (see Section~\ref{s.connected_CMA} for the definition). Some well-known examples are the pseudo-arc, the boundary of the Lakes of Wada, and solenoids.

\textbf{Solenoids.} For a sequence of primes $(p_0,p_1,\ldots)$, consider the following inverse system: for each $i\in\nat$ let $S_i$ be the circle group and $f_i:S_{i+1}\to S_i$ be the map $x\mapsto p_i\cdot x$ (mod 1), which is a continuous homomorphism. The topological group defined as the inverse limit of such an inverse system is called a \emph{solenoid}.

Solenoids were introduced by L.~Vietoris \cite{VIETORIS27} and D.~van Dantzig \cite{VANDANTZIG30}, and they have been extensively studied since then. For example, after S.~Smale's groundbreaking paper \cite{SMALE67}, R.~F.~Williams \cite{WILLIAMS69} showed that solenoids emerge as expanding attractors in dynamical systems. C.~L.~Hagopian \cite{HAGOPIAN77} proved that solenoids are the only homogeneous circle-like continua.

It is well-known \cite{AARTS91} that the solenoids associated to $(p_0,p_1,\ldots)$ and $(q_0,q_1,\ldots)$ are homeomorphic if and only if we can delete finitely many numbers from both sequences so that every prime occurs the same number of times in the remaining sequences. It is also known \cite[Thm~25.19]{HEWITT_ROSS} that every solenoid is a quotient of the so-called \emph{universal solenoid}, which is the solenoid associated to a sequence in which every prime occurs infinitely many times.

The main goal of the present paper is to show that the universal solenoid arises as the generic object in the class of connected compact metrizable abelian groups.

\begingroup
\def\thetheorem{\ref{t.main}}
\begin{theorem}
The isomorphism class of the universal solenoid is comeager in the space $\calc(\TT)$ of connected compact metrizable abelian groups.
\end{theorem}
\addtocounter{theorem}{-1}
\endgroup

We also settle the dual problem.

\begingroup
\def\thetheorem{\ref{t.generic_torsionfree}}
\begin{theorem}
The isomorphism class of $(\rat,+)$ is comeager in the space $\calt$ of countably infinite torsion-free abelian groups.
\end{theorem}
\addtocounter{theorem}{-1}
\endgroup

\begin{remark}
If we replace the circles $S_i$ with finite cyclic groups in the construction of a solenoid, the resulting inverse limit is called an \emph{odometer}. Odometers can thus be seen as 0-dimensional solenoids. This perspective is particularly useful when one compares Theorems~\ref{t.odometer} and \ref{t.main}.
\end{remark}

\begin{remark}
The reader may wonder whether the classes of connected compact metrizable abelian groups and torsion-free discrete countable abelian groups are sufficiently rich and diverse for our results to be meaningful. They are. We refer the interested reader to \cite[Chapters 12 and 13]{FUCHS}.
\end{remark}

\section{Preliminaries}
\label{s.preliminaries}

All topological groups are assumed to be Hausdorff.

\subsection{Baire category}
\label{ss.baire_category}

We recall some basic notions and facts about Baire category. All of them can be found in \cite{KECHRIS}.

A topological space $(X,\tau)$ is called \textbf{Polish} if it is separable and completely metrizable (that is, there exists a complete metric $d$ on $X$ that induces $\tau$). It is clear that countable discrete spaces are Polish. In particular, $2=\{0,1\}$ and $\nat$ with the discrete topologies are Polish. It is well-known that countable products of Polish spaces are Polish. In particular, $2^A$ and $\nat^A$ are Polish for any countable set $A$.

A subset $E$ of a topological space $X$ is \textbf{nowhere dense} if the closure of $E$ has empty interior. It is \textbf{meager} if it is a countable union of nowhere dense sets and \textbf{comeager} if its complement is meager.

\begin{theorem}[Baire Category Theorem]\label{t.bct}
In a completely metrizable space, every nonempty open set in nonmeager.
\end{theorem}

It is well-known that a subspace $E$ of a Polish space $X$ is Polish if and only if $E$ is $G_\delta$ in $X$. In particular, closed subspaces are Polish. It is easy to prove that a subset of a Polish space is comeager if and only if it contains a dense $G_\delta$ set.

\subsection{The space of connected compact subgroups}
\label{ss.space_of_subgroups}

For a metric space $(X,d)$ let $\calk(X)$ denote the set of nonempty compact subsets of $X$. It is well-known that $\calk(X)$ is also a metric space with the Hausdorff metric:
$$d_H(K,L)\defeq \inf\{\eps>0:\ K_\eps\supseteq L,\ L_\eps\supseteq K\},$$
where $A_\eps$ is the open $\eps$-neighborhood of the set $A$. It is also well-known that $(\calk(X),d_H)$ inherits several topological properties of $(X,d)$. For example, if $X$ is compact, then $\calk(X)$ is also compact (see \cite[Subsection~4.F]{KECHRIS}).

\begin{prop}
For a metric group $G$ the set
$$\calc(G)\defeq\{K\in \calk(G):\ K\text{ is a connected subgroup of }G\}$$
is closed in $\calk(G)$.
\end{prop}
\begin{proof}
We can write the set of subgroups as
$$\{K\in\calk(G):\ K\cdot K^{-1}=K\}.$$
It is easy to check that this set is closed. Also it follows from the definition of the Hausdorff metric that connected sets form a closed set in $\calk(G)$. Thus $\calc(G)$ is closed.
\end{proof}

\begin{defi}
\label{d.connected_compact_subgroups}
We call $\calc(G)$ \textbf{the space of connected compact subgroups} of $G$.
\end{defi}

Note that if $G$ is compact, then $\calc(G)$ is also compact, therefore the Baire category theorem holds in $\calc(G)$.

\subsection{Pontryagin duality}
\label{ss.pontryagin}

Here we remind the reader to the notion and fundamental properties of Pontryagin duality. Let us denote the circle group by $\circle$. For a locally compact abelian group $G$ the dual group $\what G$ is the set of all continuous homomorphisms $\chi:G\to\circle$ with pointwise addition and the compact-open topology, and $\what G$ is also a locally compact abelian group (see \cite[Chapter~1]{RUDIN}). Also recall:

\begin{theorem}[Pontryagin Duality Theorem]
\label{t.pontryagin}
If $G$ is a locally compact abelian group, then its double dual $\widehatto{G}{\widehat G}$ is canonically isomorphic to $G$ itself.
\end{theorem}

\begin{prop}
\label{p.duality}
For a locally compact abelian group $G$ the following hold:

(1) $G$ is compact metrizable $\iff$ $\what G$ is discrete countable. \cite[Corollary on page 96]{MORRIS}

(2) $G$ is compact connected $\iff$ $\what G$ is discrete torsion-free. \cite[Corollary~8.5]{HOFMANN_MORRIS}

(3) $G$ is compact torsion-free $\iff$ $\what G$ is discrete divisible. \cite[Corollary~8.5]{HOFMANN_MORRIS}

(4) For a sequence $(A_i)_{i\in\nat}$ of compact abelian groups we have $\ds\what{\prod_{i\in\nat}A_i}=\bigoplus_{i\in\nat}\what A_i$. \cite[Theorem~2.2.3.]{RUDIN}

(5) A locally compact abelian group $A$ embeds into $G$ $\iff$ $\what A$ is a quotient of $\what G$. \cite[Theorem~2.1.2.]{RUDIN}

(6) The dual of $\circle$ is $\int$. \cite[pages 12-13.]{RUDIN}

(7) The dual of the universal solenoid is $(\rat,+)$ with the discrete topology. \cite[25.4]{HEWITT_ROSS}
\end{prop}

\subsection{The space of connected compact metrizable abelian groups}
\label{ss.connected_CMA}

\begin{notation}
For a group $G$ and a set $I$ we denote by $G^{(I)}$ the direct sum $\bigoplus_{i\in I} G$.
\end{notation}

We view $\circle$ as the additive group $\real/\int$, that is, the real numbers modulo 1, and we use the usual metric $d_\circle$ (with respect to which its diameter is $\frac{1}{2}$). Clearly, the circle group is a connected compact metrizable abelian group (hereafter connected CMA group). Let $\TT$ denote the countably infinite direct product of $\circle$ with itself. Again, it is a connected CMA group and every closed connected subgroup of $\TT$ is a connected CMA group as well. We equip $\TT$ with the metric
$$d(x,y)=\sum_{n=0}^\infty\frac{1}{2^n}d_\circle(x(n),y(n)).$$
Note that, by Proposition~\ref{p.duality} (4) and (6), the dual of $\TT$ is $\int^{(\nat)}$. By (5), a CMA group $G$ is embeddable into $\TT$ if and only if the countable discrete abelian group $\what G$ is a quotient of $\int^{(\nat)}$. Since $\int^{(\nat)}$ is the free abelian group on countably infinitely many generators, every countable discrete abelian group is a quotient of $\int^{(\nat)}$, hence every CMA group is embeddable into $\TT$. Therefore, we may view $\calc(\TT)$ as the space of connected CMA groups (recall Definition~\ref{d.connected_compact_subgroups}).


\subsection{The space of multiplication tables}
\label{ss.multiplication_tables}

The space $\nnn$ of infinite tables of natural numbers is Polish (see Section~\ref{ss.baire_category}). A clopen basis for $\nnn$ consists of sets of the form
\begin{equation}\label{eq.basis1}
    \left\{A\in\nnn:\ A(n_1,m_1)=k_1,\ldots,A(n_l,m_l)=k_l\right\}
\end{equation}
with $n_i,m_i,k_i\in\nat$, ($i=1,\ldots,l$). Let us define the following subspace:
$$\calg\defeq\left\{G\in\nnn:\ G\text{ is a multiplication table and $0$ is its identity element}\right\}.$$

\begin{remark}
In the following few statements it is somewhat unnatural that we denote the identity element of a not necessarily abelian group by 0. We chose this notation because all the rest of the paper concerns abelian groups.
\end{remark}


\begin{remark}
When we consider elements of $\calg$ we use the usual shorthands of group theory. For example, to define the subspace of torsion-free groups, we write
$$\{G\in\calg:\  \forall n,k\in\nat^+\ (n^k\neq 0)\}$$
instead of the rather cumbersome
$$\{G\in\calg:\ \forall n, k\in\nat^+\ \underbrace {G(G(\ldots G(G(n,n),n),\ldots, n),n)}_{k-1\text{ times}}\neq 0\}.$$
\end{remark}

\begin{prop}\label{p.basis2}
For any finite sets $\{U_1,\dots,U_k\}$ and $\{V_1,\dots,V_l\}$ of words in $n$ variables $x_1,\dots,x_n$ and for any $a_1,\dots,a_n,b_1,\dots,b_k,c_1,\dots,c_l\in\nat$ the set
\begin{equation}
    \left\{G\in\calg:\ \bigwedge_{i=1}^k U_{i}(a_1,\dots,a_n)=b_i\land \bigwedge_{j=1}^l V_{j}(a_1,\dots,a_n)\neq c_j\right\}
\end{equation}
is clopen, and sets of this form constitute a basis for $\calg$.
\end{prop}

For a sketch of proof, see \cite[Proposition 3.5]{ELEKES22}.

In particular, we are interested in the following subspace of $\calg$:
$$\calt\defeq\left\{G\in\calg:\ G \text{ is the multiplication table of a  torsion-free abelian group}\right\}.$$

We would like to apply the Baire category theorem in $\calt$, hence we need to check that it is Polish. It is easy to see that $\calg$ is a $G_{\delta}$ subspace of $\nnn$ (see \cite[Proposition 3.1]{ELEKES} for the precise calculation). For $\calt$, observe that
$$\calt=\{G\in\calg:\ \forall n,k\in\nat\ (nkn^{-1}k^{-1}=0)\}\cap\{G\in\calg:\ \forall n,k\in\nat^{+} n^k\neq 0\}=$$
$$=\left(\bigcap_{n,k\in\nat}\underbrace{\{G\in\calg:\ nkn^{-1}k^{-1}=0\}}_{\text{clopen by Proposition~\ref{p.basis2}}}\right)\cap \left(\bigcap_{n,k\in\nat^+}\underbrace{\{G\in\calg:\ n^k\neq 0\}}_{\text{clopen by Proposition~\ref{p.basis2}}}\right)$$
is a closed subset of $\calg$, thus it is Polish.






\begin{center}
\fbox{\parbox{0.8\textwidth}{Hereafter we work only with abelian groups, therefore we use additive notation and write addition table instead of multiplication table.}}
\end{center}

\begin{notation}
    We reserve the notation $\wtilde G$ for the unique group of addition table $G$ and underlying set $\nat$.
\end{notation}

\begin{defi}\label{d.groupprop}
    A set $\calp\subseteq\calt$ is a \textbf{group property} if it is invariant under isomorphism. That is, for any $G\in\calt$, if $\wtilde{G}$ is isomorphic to $\wtilde{H}$ for some $H\in\calp$, then $G\in\calp.$ 
\end{defi}

\begin{notation}\label{n.supp}
If $\calb$ is a basic clopen set of the form
$\{G\in\calt:\ \forall i,j\leq k\ (i+j=m_{i,j})\}$ with $k\in\nat$ and $m_{i,j}\in\nat$ for all $i,j\leq k$, then let $\supp\calb\defeq\{1,\dots,k\}\cup\{m_{i,j}:\ i,j\leq k\}$.
\end{notation}

We will need the following lemma.

\begin{lemma}\label{l.dense1}
    Let $\calp\subseteq \cals\subseteq \calt$ be group properties. If for every $G\in\cals$ there exists some $H\in\calp$, such that $\wtilde{G}$ can be embedded into $\wtilde{H}$, then $\calp$ is dense in $\cals$.
\end{lemma}

As this is a special case of Lemma~\ref{l.suppinjective}, we do not prove it here.

The following remark is interesting in its own right and also proves useful in the proof of Lemma~\ref{l.suppinjective}.

\begin{remark}\label{r.natural_action}
    Let $S_\infty^\ast$ denote the set of permutations of $\nat$ that fix $0$. Then there is a natural $S_\infty^\ast\curvearrowright\calt$ action. For any $\nu\in S_\infty^\ast$ we define the induced homeomorphism $h_\nu: \calt\rightarrow\calt$ as follows: for any $G\in \calt$ the addition table $h_\nu(G)$ is defined by the equations $i+j=\nu(G(\nu^{-1}(i),\nu^{-1}(j)))$. It is an easy exercise to verify that $h_\nu$ is indeed a homeomorphism. Clearly, $\nu$ is an isomorphism between $\wtilde{G}$ and $\wtilde{h_\nu(G)}$, moreover, the orbit of any addition table is exactly its isomorphism class.
\end{remark}

\begin{lemma}\label{l.suppinjective}
    Let $\calp\subseteq\cals\subseteq\calt$ be group properties. Suppose that for every nonempty basic clopen set $\calb=\{G\in\cals:\ \forall i,j\leq k\ (i+ j=m_{i,j})\}$ there exists $G_0\in\calb$ and $H\in \calp$ and a homomorphism $\varphi:\wtilde{G_0}\rightarrow \wtilde{H}$ such that $\varphi|_{\supp \calb}$ is injective. Then $\calp$ is dense in $\cals$.
\end{lemma}

\begin{proof}
    Fix any nonempty basic clopen set $\calb=\{G\in\cals:\ \forall i,j\leq k\ (i+ j=m_{i,j})\}$. By assumption, there are $G_0\in\calb$ and $H\in \calp$ and a homomorphism $\varphi:\wtilde{G_0}\rightarrow \wtilde{H}$ that is injective on $\supp \calb$. As $\supp \calb$ is finite, there is a bijection $\nu:\nat\rightarrow\nat$ such that $(\nu\circ\varphi)|_{\supp\calb}$ is the identity function. Let $h_\nu:\calt\rightarrow\calt$ be the homeomorphism described in Remark~\ref{r.natural_action}. On the one hand, $h_\nu(H)$ is isomorphic to $H$, hence $h_\nu(H)\in\calp$. On the other hand, for any $i,j\leq k$ we have $h_\nu(H)(i,j)=\nu(H(\nu^{-1}(i),\nu^{-1}(j)))=\nu(H(\varphi(i),\varphi(j)))=\nu(\varphi(G_0(i,j)))=G_0(i,j)$, thus $h_\nu(H)$ and $G_0$ coincide on $\{0,1,\ldots k\}\times \{0,1,\ldots k\}$, which gives us $h_\nu(H)\in\calb$. Consequently, $h_\nu(H)\in \calp\cap\calb$, which completes the proof.    
\end{proof}

\begin{remark}
    Even though we do not need these stronger statements, Lemma~\ref{l.dense1} and \ref{l.suppinjective} remain true if we replace $\calg$ with $\calt$. Essentially the same proof works.
\end{remark}

\section{Genericity among connected compact metrizable abelian groups}
\label{s.connected_CMA}

In this section, we prove that there exists a comeager isomorphism class in the space $\calc(\TT)$ of connected compact metrizable abelian groups. First, let us remind the reader of the following basic fact.

\begin{fact}\label{f.ratvectorspace}
    Every countable divisible torsion-free abelian group is isomorphic to a countable dimensional $\rat$-vector space. (It is easy to check the vector space axioms.)
\end{fact}

Recall that a \textbf{continuum} is a connected compact metric space. We need two fundamental notions from continuum theory.

\begin{defi}\label{d.indecomposable}
A continuum is \textbf{indecomposable} if it cannot be written as the union of two of its proper subcontinua.
\end{defi}

\begin{defi}\label{d.composant}
In a continuum $X$, the \textbf{composant} of a point $p\in X$ is the union of all proper subcontinua of $X$ that contain $p$.
\end{defi}

Now we present a characterization of the universal solenoid.

\begin{prop}
\label{p.u_solenoid_characterized}
The universal solenoid as a topological group is characterized by the following properties:

(1) abelian

(2) compact

(3) metrizable

(4) connected

(5) torsion-free

(6) indecomposable
\end{prop}

\begin{proof}
By its construction, the universal solenoid clearly satisfies properties (1)-(4). It is well-known that it is indecomposable (see \cite[Definition~2.5 and Theorem 2.7]{NADLER92}. By Proposition~\ref{p.duality} (3) and (7), it is torsion-free.

Now let $G$ be a topological group that satisfies properties (1)-(6). By Proposition~\ref{p.duality}, properties (1)-(5) imply that $\what G$ is a discrete countable torsion-free divisible abelian group. Then, by Fact~\ref{f.ratvectorspace}, it is a countable dimensional $\rat$-vector space, hence 
 it suffices to prove the following.

\textbf{Claim.} The group $\what G$ is $1$-dimensional over $\rat$.

Suppose it is not. Then, by Proposition~\ref{p.duality} (4), $G$ can be written as $K\times L$, where $K$ and $L$ are nontrivial continua. Pick any $p=(x,y)\in K\times L$. Now the composant of $p$ in $K\times L$ is $K\times L$ itself since for any point $q=(u,v)\in K\times L$ the proper subcontinuum $(\{x\}\times L)\cup (K\times \{v\})$ of $K\times L$ contains both $p$ and $q$. However, by a well-known theorem \cite[§48, Section VI, Theorem 6]{KURATOWSKI_II}, every composant of an indecomposable continuum $\calc$ is meager in $\calc$. This is a contradiction, which proves the claim.
\end{proof}

We are almost ready to prove the main theorem. We need a lemma that roughly says that circles form a dense set in $\calc(\TT)$.

\begin{notation}
\label{n.projections}
For a number $n\in\nat$ and a group $G\in\calc(\TT)$ let $\pi_{[0,n-1]}(G)$ denote the projection of $G$ to the torus formed by the first $n$ coordinates of $\TT$. Similarly, $\pi_{[n,\infty)}(G)$ denotes the projection of $G$ to the infinite dimensional torus formed by the coordinates $n,n+1,\ldots$ of $\TT$.
\end{notation}

For every $n\in\nat^+$ consider the following subspace of $\calc(\TT)$:
$$\calf_n\defeq\{G\in\calc(\TT):\ \pi_{[0,n-1]}(G)\text{ is a circle and } \pi_{[n,\infty)}(G)\text{ is the trivial group }\}.$$

\begin{lemma}
\label{l.dense_in_C}
For every $n\in\nat^+$ the subspace $\calf_n$ is a $\frac{1}{2^{n-1}}$-net in $\calc(\TT)$.
\end{lemma}

\begin{proof}
Fix a group $G\in\calc(\TT)$. Let $T\defeq\pi_{[0,n-1]}(G)$. Then $d(G,T\times\{0\}\times\{0\}\times\ldots)\leq\frac{1}{2^n}$. Also note that $T$ is a compact connected subgroup of $\circle^n$, hence, by Cartan's theorem, it is a Lie group. Every compact connected abelian Lie group is a torus, hence $T\cong\circle^k$ for some $k\leq n$.

\textbf{Case 1.} $k=0$. Then let $T_n$ be the unique element in $\calf_n$ such that its projection to every coordinate is $\{0\}$, except for the $(n-1)$th coordinate, where it is $\circle$. We have $d(G,T_n)\leq\frac{1}{2^{n-1}}$ because $\pi_{[0,n-2]}(G)=\pi_{[0,n-2]}(T_n)$.

\textbf{Case 2.} The group $T$ is a nondegenerate torus. It suffices to approximate $T$ by its circle subgroups.

\textbf{Claim.} We may assume that $T=\circle^k$ (with the sum metric). Indeed,  pick any isomorphism $\circle^k\to T$. This is uniformly continuous since $\circle^k$ is compact. Then for every $\eps>0$ there is $\delta>0$ such that the image of a $\delta$-net of $\circle^k$ is an $\eps$-net of $T$.

Fix any $\delta>0$. Let $N>\frac{\sqrt k}{\delta}$. It is straightforward to check that the finite subgroup generated by the element $\left(\frac{1}{N},\frac{1}{N^2},\ldots,\frac{1}{N^k}\right)$ is a $\delta$-net in $\circle^k$ (it meets every cube of the $\frac{1}{N}$-grid on $\circle^k$). It is well-known that in a connected compact Lie group every element is contained in a 1-parameter subgroup, which is isomorphic to either $\real$ or $\circle$. Now let $H\leq\circle^k$ be a 1-parameter subgroup that contains the torsion element $\left(\frac{1}{N},\frac{1}{N^2},\ldots,\frac{1}{N^k}\right)$. Then $H$ is isomorphic to $\circle$ and it is a $\delta$-net in $\circle^k$, which concludes the proof.
\end{proof}

\begin{theorem}
\label{t.main}
The isomorphism class of the universal solenoid is comeager in the space $\calc(\TT)$ of connected compact metrizable abelian groups.
\end{theorem}

\begin{proof}
Since properties (1)-(4) of Proposition~\ref{p.u_solenoid_characterized} are automatic in $\calc(\TT)$, it suffices to show that the generic $G\in\calc(\TT)$ is torsion-free and indecomposable.

\textbf{Claim.} The isomorphism class of the universal solenoid is dense in $\calc(\TT)$.

By Lemma~\ref{l.dense_in_C}, it suffices to find universal solenoids $\frac{1}{2^n}$-close to every element of $\calf_n$ for any $n\in\nat^+$. This is easy: if we are given a circle group $S_0\leq\circle^n$, we construct our universal solenoid starting with that circle. That is, for each $i\geq 1$ let $S_i$ be the circle group that is the $(n-1+i)$th factor of $\TT$. Now we can construct a universal solenoid inside $\TT$, whose projection to the first $n$ coordinates is $S_0$. This is a $\frac{1}{2^n}$-approximation, which proves the claim.

It remains to prove that torsion-free groups and indecomposable groups form $G_\delta$ sets in $\calc(\TT)$. The latter is known, see \cite[Remark 5, page 207]{KURATOWSKI_II}.

Let us define  $f_m:\TT\rightarrow \TT$ as $f_m: x\mapsto m\cdot x$. Clearly, $f_m$ is continuous and the set of torsion-free groups is
$$\left\{G\in \calc(\TT):\ \forall m,k\in \nat^+\  \left(0\notin f_m{\left[G\setminus B{\left(0,\tfrac{1}{k}\right)}\right]}\right)\right\}=$$
$$=\bigcap_{m,k\in\nat^+} \underbrace{\left\{G\in \calc(\TT):\ G\cap \underbrace{B{\left(0,\tfrac{1}{k}\right)}^c\cap f_m^{-1}(\{0\})}_{\text{closed}}=\emptyset\right\}}_{\text{open}},$$
which is $G_\delta$.
\end{proof}

\section{Genericity among torsion-free discrete countably infinite abelian groups}\label{s.torsion-free_countable}

In this section, we prove the dual of the result of Section~\ref{s.connected_CMA}, that is, there exists a comeager isomorphism class in the space $\calt$ of torsion-free countably infinite abelian groups.




Let us explain in what sense Section~\ref{s.torsion-free_countable} is the dual of Section~\ref{s.connected_CMA}. In Theorem~\ref{t.main}, we proved that the isomorphism class of the universal solenoid is comeager in the space of connected compact metrizable abelian groups. By Proposition~\ref{p.duality} (1) and (2), the Pontryagin dual of the property \textit{connected compact metrizable} is \textit{torsion-free discrete countable}. Note that since we study only torsion-free groups, by neglecting the trivial group, we may conveniently replace the word countable with countably infinite, which fits the setting introduced in Subsection~\ref{ss.multiplication_tables}. Thus the dual problem is whether there is a comeager isomorphism class in the space $\calt$ of torsion-free countably infinite abelian groups. A natural candidate is the Pontryagin dual of the universal solenoid, which is $(\rat,+)$ with the discrete topology by Proposition~\ref{p.duality} (7).

First, we will prove that divisible groups form a comeager set in $\calt.$ The following well-known fact will be useful.

\begin{fact}\label{f.embedrat}
    Every countable torsion-free abelian group can be embedded into $\rat^{(\nat)}$.
\end{fact}

\begin{proof}
By \cite[Chapter~4., Theorem~1.4]{FUCHS} and \cite[Chapter~4., Theorem~3.1]{FUCHS}, every countable abelian group can be embedded into a group of the form $T \oplus \rat^{(\nat)}$, where $T$ is a torsion group. If be embed a torsion-free group into $T\oplus \rat^{(\nat)}$, then the projection to $\rat^{(\nat)}$ is injective, thus we get an embedding into $\rat^{(\nat)}$.
\end{proof}

\begin{lemma}\label{l.divisible}
The set
$\cald\defeq\{G\in\calt:\ \wtilde G\text{ is divisible}\}$ is dense $G_\delta$ in $\calt$ and thus comeager (see Section ~\ref{ss.baire_category}).
\end{lemma}
\begin{proof}
    The proof is essentially the same as \cite[Proposition 3.21]{ELEKES22}, but as it is a relatively short proof, we reproduce it here. First, we prove the $G_\delta$ part. By definition:
$$\cald=\{G\in\calt:\ \forall n, k\in \nat^{+}\ \exists m\in\nat\ (k\cdot m=n)\}=\bigcap_{n,k\in\nat^{+}} \bigcup_{m\in\nat}\ \underbrace{\{G\in\calt:\ k\cdot m=n\}}_{\text{clopen in $\calt$ by Proposition~\ref{p.basis2}}},$$
which is $G_\delta$ in $\calt$.

For the density, we know by Fact~\ref{f.embedrat} that every countable torsion-free abelian group can be embedded into $\rat^{(\nat)}$, which is a countably infinite divisible torsion-free abelian group. Thus Lemma~\ref{l.dense1} completes the proof.
\end{proof}

\begin{theorem}\label{t.generic_torsionfree}
    The isomorphism class of $(\rat,+)$ is comeager in the space $\calt$ of countably infinite torsion-free abelian groups.
\end{theorem}
\begin{proof}
Since Lemma~\ref{l.divisible} says that divisible groups form a comeager subspace $\cald\subseteq\calt$, by Fact~\ref{f.ratvectorspace}, it suffices to prove that for the generic $G\in\cald$ we have $\dim_{\rat}G=1$. That is, we need to show that
$$\calo:=\{G\in \cald:\ \forall k,l \in \nat^{+} \exists \ m,n\in \nat^{+} (m\cdot k=n\cdot l)\}$$
is comeager in $\cald$.
It is easy to see that this set is $G_\delta$ in $\cald$, as we can write it as:
$$\bigcap_{k,l\in\nat^{+}}\ \bigcup_{m,n\in\nat^{+}}\ \underbrace{\{G\in\cald:\ m\cdot k= n\cdot l\}}_{\text{clopen in $\cald$ by Proposition~\ref{p.basis2}}}.$$
Now we prove the density. We would like to apply Lemma~\ref{l.suppinjective} for the group properties $\calo\subseteq\cald\subseteq\calt$. So let us fix any nonempty basic clopen set $\calb\subseteq\cals$. Pick any $G_0\in\calb$. As $\wtilde{H}\cong \rat$ for any $H\in\calo$, it is enough to find a homomorphism $\varphi:\wtilde{G_0}\rightarrow \rat$ such that $\varphi|_{\supp{\calb}}$ is injective. By Fact~\ref{f.embedrat}, there is an embedding $\nu:\wtilde{G_0}\rightarrow \rat^{(\nat)}$. Then there are only finitely many nonzero coordinates of the elements of $\nu(\supp\calb)$, let $\{\frac{a_i}{b_i}\}_{i=1}^n$ be an enumeration of them. We can suppose that $b_i=1$ for all $1\leq i\leq n$ since otherwise we can replace $\nu$ by $\nu':=N\cdot\nu$, where $N=\prod_{i=1}^n b_i$. Fix an integer $K$ such that $K>2|a_i|$ for every $1\leq i\leq n$. Now let us define $\varphi:\wtilde{G_0}\rightarrow\rat$ as follows:
$$\varphi(g):= \sum_{i=0}^\infty K^{i}\cdot \nu(g)_i,$$
 where $\nu(g)_i$ denotes the $i$th coordinate of $\nu(g)$. This sum is well-defined for all $g\in \wtilde{G_0}$ because every element of $\rat^{(\nat)}$ has finitely many nonzero coordinates. It is easy to check that $\varphi$ is a homomorphism. Now it remains to prove that $\varphi|_{\supp\calb}$ is injective. Take any $g,h\in\supp\calb$ such that $\varphi(g)=\varphi(h)$. We will prove by induction that $\nu(g)_n=\nu(h)_n$ for all $n$, which implies that $g=h$. For $n=0$ notice that  
 $$0=\varphi(g)-\varphi(h)=\sum_{i=0}^\infty K^{i}\cdot (\nu(g)_i-\nu(h)_i)\equiv \nu(g)_0-\nu(h)_0\pmod{K}.$$
But due to our assumption on $K$, we know that $|\nu(g)_0-\nu(h)_0|<K$ and thus $\nu(g)_0=\nu(h)_0$. For the inductive step, assume that $\nu(g)_i=\nu(h)_i$ for all $i\leq n-1$. Then:
$$0=\sum_{i=0}^\infty K^{i}\cdot (\nu(g)_i-\nu(h)_i)\equiv \sum_{i=0}^{n-1} K^{i}\cdot (\nu(g)_i-\nu(h)_i) = K^{n-1}(\nu(g)_n-\nu(h)_n)\pmod{K^{n}},$$
where the last equality holds because of the inductive hypothesis. But then $\nu(g)_n-\nu(h)_n)\equiv 0\pmod{K}$, from which we can conclude again that $\nu(g)_n=\nu(h)_n$. Thus $\nu(g)_n=\nu(h)_n$ for all $n$, which completes the proof. 
\end{proof}

\section{Open questions}

As we have mentioned in the introduction, the universal odometer, which can be viewed as a 0-dimensional solenoid, appears in the description of the generic compact metrizable abelian group (see \cite[Section~4]{ELEKES22}). In the present paper, we proved that the (1-dimensional) universal solenoid is the generic connected compact metrizable abelian group. There are higher dimensional solenoids which are compact metrizable abelian groups.

\begin{question}
Is there a natural category where the generic object is a higher dimensional solenoid (or the product of higher dimensional solenoids)?
\end{question}

For a thorough survey on solenoids, see \cite{VERJOVSKY22} by A.~Verjovsky.

\section*{Acknowledgement}

We would like to thank Olga Lukina for the helpful discussions and the reference she pointed out.

This research was supported by the National Research, Development and Innovation Office -- NKFIH, grant no. 124749.
The second author was also supported by the National Research, Development and Innovation Office -- NKFIH, grant no.~129211. The third, fourth and fifth authors were also supported by the Hungarian Academy of Sciences Momentum Grant no. 2022-58. The fourth author was supported by the ÚNKP-22-1 New National Excellence Program of the Ministry for Innovation and Technology from the source of the National Research, Development and Innovation Fund.

\bibliographystyle{acm}
\bibliography{biblio}

\end{document}